\documentclass{article}

\usepackage{PRIMEarxiv}
\usepackage[utf8]{inputenc} 
\usepackage[T1]{fontenc}    
\usepackage{hyperref}       
\usepackage{url}            
\usepackage{booktabs}       
\usepackage{amsfonts}       
\usepackage{nicefrac}       
\usepackage{microtype}      
\usepackage{lipsum}
\usepackage{fancyhdr}       
\usepackage{graphicx}       
\graphicspath{{media/}}     

\usepackage{lineno,hyperref}
\usepackage{amsfonts}
\usepackage{amsmath}
\usepackage{amsthm}
\usepackage{cases}
\usepackage[shortlabels]{enumitem}

\usepackage{bbm}

\usepackage{tikz,pgf,pgfplots}

\usetikzlibrary{arrows,shapes,snakes,automata,backgrounds,petri,fit,intersections,quotes,positioning}

\usepackage{tikz-3dplot}

\pgfplotsset{compat=1.14}

\usepackage{graphicx}
\usepackage{wasysym}
\usepackage{dsfont} 
\usepackage{mathrsfs}
\usepackage{stmaryrd}
\usepackage{multicol}
\usepackage{cite}
\graphicspath{{graphs/}}

\newcommand{\R}{{ \mathbb  R  }}
\newcommand{\C}{  \mathbb  C }

\newcommand{\N}{  \mathbb N }

\renewcommand{\l}{\left\langle}
\renewcommand{\r}{\right\rangle}

\newcommand{\dsize}{\displaystyle}
\renewcommand{\cal}{\mathcal}
\numberwithin{equation}{section}

\newcommand{\G} {\mathcal G} 
\newcommand{\F}{\cal F}
\renewcommand{\H}{\mathcal{H}}
\renewcommand{\P}{\mathcal{P}}

\newtheorem{Thm}{Theorem}[section]
\newtheorem{Lemma}[Thm]{Lemma}
\newtheorem{Cor}[Thm]{Corollary}
\newtheorem{Prop}[Thm]{Proposition}

\theoremstyle{plain}
\newtheorem{theorem}{Theorem}[section]

\newtheorem{lemma}[Thm]{Lemma}

\theoremstyle{definition}
\newtheorem{definition}[theorem]{Definition}

\theoremstyle{remark}

\definecolor{cof}{RGB}{219,144,71}
\definecolor{pur}{RGB}{186,146,162}
\definecolor{greeo}{RGB}{91,173,69}
\definecolor{greet}{RGB}{52,111,72}

\title{Applications of  Lax-Milgram   theorem  to problems in frame theory }

\author{
  Laura De Carli \\
  Florida International Univ., Dept. of mathematics and statistics, Miami, FL 33199.
  \\
  \texttt{decarlil@fiu.edu} \\
   \And
  Pierluigi Vellucci \\
  Univ. Roma Tre,  Dept.  of economics,
 Via Silvio D'Amico 77, 
  00145 Rome, Italy.
  \\ 
  \texttt{pierluigi.vellucci@uniroma3.it} \\
}

\begin{document}
\maketitle

\begin{abstract}
	We apply Lax-Milgram theorem  to characterize scalable and piecewise scalable frame  in finite and infinite-dimensional Hilbert spaces. We also 
	introduce a   method for approximating   the inverse  frame operator  using finite-dimensional linear algebra which, to the best of our knowledge, is new in the literature.
\end{abstract}



\section{Introduction}
\label{intro}

 
Let  $J\subset \N$ denote a set of indices that can  either be finite or infinite. 
A sequence of distinct vectors ${\cal F}=\{x_j\}_{j\in J} $   belonging to a separable Hilbert space $(\mathcal{H}, \ \l\ ,\ \r)$  is a \emph{ frame}  for $\H$ if there exist positive constants $A$, $B>0$ such that
\begin{equation}
\label{eq3}
A\|x\|^2\leq\sum_{j\in J} |\l x,x_j\r|^2\leq B\|x\|^2 
\end{equation}
for every $x\in \H$.  We will assume, often without saying, that the vectors $x_j$  have norm $1$.
The  \emph{frame operator} of $\F$ is $S=S_{\F}:\mathcal{H}\to \mathcal{H}$,   $ S(x)=\sum_{j\in J} \l x,x_j\r x_j.$ 
From general frame theory, we know that 
$S$ is bounded, self-adjoint  and invertible, and   
  the representation formula \begin{equation}\label{invframe}
	x=\sum_{j\in J}  \langle x,\,S^{-1}x_j\rangle x_j,  
\end{equation} holds  for every $x\in \H$.
 Furthermore, the series in \eqref{invframe}  converges unconditionally for all $x\in\mathcal{H}$.  The scalars $\langle x,\,S^{-1}x_j\rangle$ are called the \emph{frame coefficients} of $x$ relative to   the frame $\F$.


In general, the representation formula \eqref{invframe}  may difficult to apply because  inverting $S$  can be challenging, but if    $\F$ is a  {\emph Parseval  frame} (i.e., if $A=B=1$), then   $S x= S^{-1}x=x$,   and    the  reconstruction formula  $	x= \sum_{j\in J}  \langle x,\,  x_j\rangle x_j, $ holds for every $x\in\H$.   A key question in frame theory is how to  modify a given frame so that the resulting system forms a Parseval frame.  One way to do this  is just by scaling each frame vector in such a way to obtain a Parseval frame.  Frames for which such modification is possible are  called \emph{scalable}.  
 Unfortunately frame scaling is a very difficult  problem  \cite{CC,CCH,CK,CKO,CKL,DK,KO,KOP}.  A more general definition of scaling  
 is given in  \cite{CDT}.  We have recalled these definitions and  properties of scalable frames in Section 2.2.

 In this paper we use the Lax-Milgram  theorem to   characterize  frames that are scalable in the classical and generalized sense.  We also present   a method  for approximating   the  inverse frame operators which seems  new in the literature. Our  main results are in Sections 3 and 4. In Section 2 we have collected some preliminaries on  frames and the Lax Milgram theorem,  and in Section 5  we have provided some examples.
 
\section {Preliminaries}

\subsection{Basic on frames}

Most of the results presented in this section can be found e.g. in the classical  textbook   \cite{Chr}.
 For a given nonempty set $J\subset \N$   we denote with  $\ell^2(J)$  the space of sequences of real  numbers $\vec a=(a_j)_{j\in J}$ for which $\sum_{j\in J} |a_j|^2 <\infty$ and with   $\ell^\infty(J)$  the space of sequences   for which $\sup_{j\in J} |a_j| <\infty$.
 When convenient, we will  also use the notation $a=(a(j))_{j\in J}$ to  denote vectors in $\ell^p (J)$.

Let  $\F=\{x_j\}_{j\in J} $   be a frame for  a separable Hilbert space $\H$.  The 
\emph{synthesis operator   }    of $\F$ is $$T_\F:\ell^2(J)\to\H,\qquad 
  T_\F(\vec a)=\sum_{j\in J} a_jx_j.$$ 
The    \emph{ analysis   operator} of $\F$ is   
 $$T^*_{\F} :\H\to \ell^2(J),\qquad T^*_\F(x)=(\l x, x_j\r)_{j\in J}.$$    
 We will omit the subscript $\F$ when there is no ambiguity. Note  that $T  T^* =S $,    the  frame operator  of $\F$.
The operator  $S$  is self-adjoint, invertible  and positive, and satisfies.
\begin{equation}\label{eq66}
  A||x||^2\leq \l S(x), x\r=\sum_{j\in J}  |\l x,x_j \r|^2 =   || T^*(x)||^2\leq B||x||^2
 \end{equation}
where   $0<A <B$ are the frame constants of $\F$.   If $\F$ is a Parseval frame, then 
 $\l S(x), x\r=  ||x||^2$, and $S(x)=x$.
 
 When $\H $ is a finite-dimensional vector space, we can identify $\H$ with $\R^n$  and  the vectors of the frame with vectors in $\R^n$ that, with some abuse of notation, we will still denote with $x_j$.   The synthesis operator of $\F$ is represented by  a matrix $M_T$ whose   $i$-th column is $x_i$, in the sense   that $M_T x=T(x)$. The frame operator of $\F$ is represented by the matrix 
 $M_S=M_T (M_{ T })^*$.
 The element of the matrix $M_S=\{m_{i,k}\}_{i,k\leq n}$ are the dot products of the rows of $M_T$, i.e. 
  \begin{equation}\label{mij} 
  	m_{i,k}=\sum_{j\in J} x_j(k)x_j(i).
  \end{equation}
Note that the   diagonal elements of $M$ are  positive. Indeed, $m_{k,k}=\sum_{j=\in J} x_j(k)^2 \ge 0. $   If   $x_j(k)=0$ for every $j$, then the vectors $x_j$ would not span $\R^n$; thus, $x_j(k)\ne 0$ for some $j$ and $m_{k,k}>0$.
 
Frame operators of frames  in $\ell^2(J)$  can be represented by    matrices also when $|J|=\infty$.  We will not discuss the properties of  this representation; the reader can  refer  to   \cite{B}  for a thorough discussion on  representations of frame operators  with matrices in  infinite-dimensional spaces.   

 When convenient, we  will  identify  $k$-dimensional sub-spaces of $\ell^2(J)$ with $\R^k$; for example, if 
 $\Pi_k:\ell^2(J)\to \ell^2(J)$,  $\Pi_k(x)=(x(j_1)...\, x(j_k) ,\,0,... )$      is the orthogonal projection on the first $k$ components of $x$, 
 we will identify vectors $  (x(j_1) \, ... ,\, x(j_k) ,\,0,... )\in \ell^2(J) $ with  vectors $ (x(j_1) \, ... ,\, x(j_k)), \ x(j)\in\R^k$.

 	Let $\F=\{x_j\}_{j\in J}$ be a frame on a Hilbert space $\H$ with frame operator $S$.  
 	Let $P:\H\to\H$ be an orthogonal projection.   Recall that an orthogonal projection $P:\H\to\H$ satisfies  $P^*=P\circ P=P$.  	We can see at once that   
 	$$ PSP(x)= P\left(\sum_{i\in J} \l x_j, Px\r x_j\right)=\sum_{i\in J} \l Px_j, Px\r Px_j
 	$$ which is the  frame operator of $P(\F)$.
 This observation proves the following
 
 	\begin{Prop}\label{PF}
 The set $P(\F)=\{Px_j\}_{j\in J}$ is a frame   on $P(\H)$ with frame operator $PSP$. 
 \end{Prop}

Note that the operator $PSP$  is defined in $\H$, but it is only invertible in $P(\H)$. In the following we will assume, often  without saying, that $PSP:P(\H)\to P(\H)$.  

 \noindent
 {\it Remark.}	If $\H=\ell^2(J)$  and $P=\Pi_k$,  we can identify $\Pi_k(\ell^2(J))$ with $\R^k$ and the matrix that represents $\Pi_k S\Pi_k$  on $\Pi_k(\ell^2(J))$  with a $k\times k$ matrix.   In view of this observation and  \eqref{mij}, we can see at once that if $S$  is represented by the matrix $M_S$, 
 	the frame operator of  $\Pi_k(\F)$ is represented  on $\Pi_k(\ell^2(J))$ by the matrix obtained from  the intersection of the first $k$ rows and columns of $M_S$.
 	Note that by the Cauchy interlacing theorem  (see e.g. \cite{H}, Thm. 4.3.17)  principal sub-matrices of  positive-definite  symmetric matrices are always  positive-definite.   

  
 	
 	%

 

\subsection{Scalable and piecewise scalable frames}
Let us recall the definition of scalable frames from \cite{KOP}.

\begin{definition}\label{def-scalable}
	A frame  $\F=\{x_i\}_{i\in J} $  for a Hilbert space $\H$  is
	{\it scalable} if there exists     $\vec c \in \ell^\infty(J)$ for which  $\F_{\vec c}=:\{c_i x_i\}_{i\in J}$ is a Parseval frame.
	 
\end{definition}
 We denote with  $T_{\vec c}$, $T^*_{\vec c}$ and $S_{\vec c}$  the  analysis,  the synthesis  and  the frame operator  of $ \F_{\vec c}$.  Thus,
 \begin{equation}\label{ss}
 	S_{\vec c}\,x= \sum_{j\in J} |c_j|^2 \l x, x_j\r x_j
 	\end{equation}

Some of the $c_j$  are allowed to be zero as long as  $\F_{\vec c}$ is still a frame; we say that  $\F$ is {\it strictly scalable} is if can be scaled to a Parseval frame with nonzero coefficients $c_j$.

 We let   $D_{\vec c}:\ell^2(J)\to \ell^2(J)$ be the   operator  defined as  \begin{equation}\label{dil} D_{\vec c}(\vec y) = (y_j c_j)_{j\in J}.\end{equation}
 We also let 
	$T^*_{\vec c} :\H\to\ell^2(J)$,    
	$$T^*_{\vec c}(x)= D_{\vec c}\circ T^*= (c_j\l x,\,  x_j\r)_{j\in J}= ( \l x,\,  c_jx_j\r)_{j\in J}  $$  be the analysis operator of $\F_{\vec c}$,
	 and   
	 \begin{equation}\label{Sc}
	 	S_{\vec c}= (T^*_{\vec c})^*\circ T^*_{\vec c} = T\circ D_{\vec c}^2\circ T^*\end{equation}
 	be the frame operator of $\F_{\vec c}$. Thus, 
	   $\F$ is scalable if and only if  $ T\circ D_{\vec c}^2\circ T^* =I$, the identity operator on $\H$.
	

The following generalization of  Definition \ref{def-scalable} is in \cite{CDT}.  
\begin{definition}\label{def-scalable2} 
	A frame $\F=\{x_i\}_{i\in J} $ for a Hilbert space $\H$ is {\emph  piecewise
		scalable }  if  there exist an orthogonal  projection $P:\H\to\H$ and   constants $\{a_i, b_i\}_{i\in J} $
	so that $\F_{ \vec a, \vec b}=:\{a_iPx_i+b_i(I -P)x_i\}_{i\in J}$ is a Parseval frame for $\H$.  
 \end{definition}
  In \cite{CDT} only   finite frames in $\R^n$ are considered, but    the definitions and  many results proved  in that paper   easily  generalize  to frames in infinite dimensional spaces. 
   
%
With the notation previously introduced, the analysis operator   of  $\F_{ \vec a, \vec b}$ is  
\begin{align*}
 T_{ \vec a, \vec b}^*x &=: (\l x,\, a_jPx_j + b_j(I-P)x_j\r)_{j\in J} 
 \\ &= (\l Px,\, a_j x_j \r  )_{j\in J}+
 (\l (I-P)x,\, b_j x_j \r  )_{j\in J}
\\ &=D_{\vec a}T^*P x+ D_{\vec b}T^*(I-P) x.\end{align*} 
The synthesis operator is $T_{  \vec a, \vec b} (\vec c)=P T D_{\vec a}\,(\vec c)+ (I-P)T D_{\vec b}\, (\vec c)$,
and the frame operator is 
\begin{align}\nonumber
S_{ \vec a, \vec b} &=T_{  \vec a, \vec b} T^*_{ \vec a, \vec b}=  (P T D_{\vec a} + (I-P)T D_{\vec b} )(D_{\vec a} T ^*P + D_{\vec b} T ^*(I-P))\\
\nonumber
&=   PT D_{\vec a}^2 T ^* P  +(I-P) T D_{\vec b}^2 T ^* (I-P)+   \{R+R^*  \}
\end{align}
where $R= (I-P) T D_{\vec b}  D_{\vec a}T ^*P$.
In view of \eqref{Sc}, we  can write 
\begin{equation}\label{f2} S_{  \vec a, \vec b}=  PS_{\vec a} P   + (I-P)S_{\vec b} (I-P) +  \{R+R^*  \}.
\end{equation}

We prove the following 

\begin{Lemma}\label{pf} With the notation and definitions previously introduced, a  frame $\F$ is piecewise scalable with projection $P$  and scaling constants $\vec a, \vec b$ if and only if the following holds for every  $u, w\in \H$
	
	a) $ 	\l   D_{\vec a}T ^*Pu, \  D_{\vec b}T ^*(I-P)w\r =0,$  

   b)  $|| D_{\vec a}T ^*Pu||^2+ || D_{\vec b}T ^*(I-P)u||^2=||u||^2$.
\end{Lemma}

\begin{proof}
	By Theorem 2.8 in  \cite {CDT}, $\F$  is piecewise scalable if and only if 
	i)  and ii) below hold.
	$$ i) \quad R=(I-P) T D_{\vec b}  D_{\vec a}T ^*P=0, \qquad  ii) \quad PS_{\vec a} P  + (I-P)S_{\vec b}  (I-P) = I.
	 .
	$$
The proof of Theorem 2.8 is  for frames in $\R^n$, but its   generalization  to infinite-dimensional vector spaces is straightforward.

Let us show that a) is equivalent to i) and b) to ii).  Indeed,  for every $u, w\in \H$, 
$$ 
\l Ru,\, w\r= \l (I-P) T D_{\vec b}  D_{\vec a}T ^*Pu, \, w\r= \l   D_{\vec a}T ^*Pu, \  D_{\vec b}T ^*(I-P)w\r  $$    thus $R=0$ if and only if    a)   holds.

Let $S_{\vec a, \, \vec b}=PS_{\vec a} P  + (I-P)S_{\vec b}  (I-P)$. 
Since $P$ and $I-P$ have orthogonal range,   
$$\l  S_{\vec a, \, \vec b}u\, u\r= \l PS_{\vec a} Pu, u\r+ \l    (I-P)S_{\vec b} (I-P) u, u\r= || D_{\vec a}T ^*Pu||^2+ || D_{\vec b}T ^*(I-P)u||^2.  
$$  If  ii) holds, then $  S_{\vec a, \, \vec b} = I$   and  b) trivially holds. Conversely, if b) holds, then 
  $\l  S_{\vec a, \, \vec b} u  ,\, u\r=||u||^2 $ and  this is only possible if $S_{\vec a, \, \vec b}  =I$.  

	\end{proof}

\subsection{The Stampacchia and the  Lax-Milgram theorems}

We recall the following definitions.
A bilinear form  $a=a(\cdot,\cdot):\H\times \H \to\R$  is \emph{bounded  }    if  there exist a constant $ C>0$ for which 
$   a(u,v)\leq C\|u\| \|v\|$ whenever $u,v\in \mathcal{H}$. 
We say that  $a$   is {\it coercive} if  there exist a constant $c>0$ for which 
$ a(u,u)\ge c||u||^2 $ for every $u\in \H$, and   it is     symmetric  if $a(u,v )= {a(v,u)}$.

A linear  function  $L:\H\to\R$  is \emph{ bounded} (or continuous)   if there exists a constant $C>0$ such that $|L(v)|\leq C||v||$ for  every $v\in \H$. We can also say that $L\in \H'$, the dual space of $\H$. By the Riesz representation theorem,  for every $L\in \H'$ there exists a unique  $f\in H$ such that $L(u)=\langle f, u\rangle$ for every $v\in\mathcal{H}$.

 Given a  bilinear form $a(\cdot,\, \cdot)$ on $\H$  and  a  function  $L =\l f, \cdot\r \in \H'$, we define the  functional $J:
\H\to \R$,  
\begin{equation}\label{def-J} J(u)=\frac{1}{2}a(u,u)-  Lu
\end{equation}

   The Lax-Milgram theorem can be stated as follows:
 \begin{theorem}[Lax-Milgram]
 	\label{lm:lax}
 	Let  $a(\cdot,\, \cdot) $  be a  bilinear, bounded and coercive form on $\H$ and let $L \in \H'$. Let $ J  $ be  as  in \eqref{def-J}.  
 There exists a unique $v \in \H$   for which  $a(u,\,v)=L(u)$ for every $u\in \H$.  
 
 If $a$ is symmetric,    $v$ is characterized by the property
 \begin{equation}\label{char}
 	J(v)=\min_{u\in \H}\, J(u).
 \end{equation}
 	 where $ J  $ is   as  in \eqref{def-J}.
 \end{theorem}

The following theorem due to G. Stampacchia  is an important generalization of the Lax-Milgram theorem
 
 \begin{theorem}[Stampacchia]
 	\label{Stampacchia}
 	Let $K$ be a nonempty closed convex subset of   $\H$.  Let  $a(\cdot,\, \cdot)  $   be  as in Theorem \ref{lm:lax}. Then, for  every $f\in\H$,  exists a unique $v\in K$   such that
 	 \begin{equation}\label{st} a(v,u-v)\ge \l f,\,  u-v\r \quad \mbox{ for every   $u\in K$. }
 		\end{equation} 
 	If $a$ is conjugate-symmetric,   $v$ is characterized by the property
 	$$
 	J(v)=\min_{u\in K} J(u)
 	$$ 
 	where $ J  $ is   as  in \eqref{def-J}.

 \end{theorem} 
For the proof of  Theorems \ref{lm:lax} and \ref{Stampacchia} see e.g. \cite{Brezis},  section 5.3.

Theorem \ref{lm:lax}  can be  proven by showing that there exists  a linear,  bounded, bijective   application $A:\H\to\H$ for which 
$ a(u,v)= \langle Au,\, v\rangle$    whenever $u, v\in \H$.
If $\H$ is finite-dimensional,  the functional $A$ is represented by a matrix   that can still be denoted with $A$.
Thus,  for a given $f\in\H$,  the unique  $ v$ in Theorem \ref{lm:lax}  satisfies  $\langle Av, u\rangle=\l f, u\r$ for every $u\in \H$,  and  hence   $ v=A^{-1}f$.   
If $a$ is  symmetric, then $A$ is self-adjoint, and by the second part of Lax-Milgram theorem, 
 the minimum of the functional $J(u)=\frac 12\l  Au, u\r-\l f,u\r$ is attained when $u=v=A^{-1}f$.

 If $K$ is a subspace of $\H$, then   for  a given  $f\in\H$,  the  unique $v\in K$  that satisfies \eqref{st}   satisfies also  
 $ a(v,u)= \l Av, u\r= \l f,u\r$    for every $u\in K$. The properties of such $v$ are described by the following
 

 \begin{Cor} 
	\label{lm:lax2}

	Let  $A:\H\to\H$  be  linear,  bounded and  invertible. 	Let $K$ be a nonempty closed   subspace of     $\H$.  
	
 a)  For a given $f\in\H$  exists a unique $v\in K$ for which $ \langle Av, u \rangle=\l f, u\r$ for every $ u\in K$.
If $A$ is  self-adjoint,      the minimum of the operator  
 \begin{equation}  \label{equivJ} J(u)=\frac 12\l  Au, u\r-\l f,u\r 
	  \end{equation}
is attained  when $u=v$.

b) Let $P=P_K  :\H\to \H$ be the orthogonal projection on   $K$.   	The operator  $PAP$ is invertible on $K$,
 and  for every $f\in \H$,  the element  $ v= ( PAP)^{-1} Pf $ is characterized by the property  \eqref{equivJ}.
	  
	\end{Cor}

\begin{proof}
  a) follows directly from Theorems Theorems \ref{lm:lax}   and \ref{Stampacchia}.
	
	Let us  prove b).
  For a given $f\in\H$, the element $v $ in \eqref{equivJ} is such that  $v=Pv$. By   a),  for every $u\in \H$ we have that  
	$\l A  Pv, \, Pu\r=\l PAP v, u\r= \l Pf,   u\r$. Thus,  for a given $f \in P(\H)$,   we can always find a unique $v\in P(\H )$ for which   $  PAPv= Pf  $,   which proves that  the operator  $PAP$ is invertible on $P(\H )$.
%
\end{proof}

We state and prove here  an easy lemma  that will be useful in  the following sections

\begin{Lemma}\label{el} Let $F:\H\to\R$; suppose that there exist   $x_0\in\H$ for  which    $F(x_0)=\min_{x\in\H} F(x)$. Then 
	$$F(x_0)=\inf_{r>0}\min_{||y||=r} F(y).$$
\end{Lemma}

\begin{proof}   
	By assumption,  $ F(x_0)\leq \min_{||y||=r} F(y)$ for every $r>0$, and so also  $ F(x_0)\leq \inf_{r>0}\min_{||y||=r} F(y)$. On the other hand, for every $x\in\R^n$ we have that $F(x)\ge \min_{||y||=||x||} F(y)$, and so 
	$ \min_{x\in\R^d} F(x)\ge \inf_{||x||>0}\min_{||y||=||x||} F(y)$.  Thus, $ F(x_0)\ge \inf_{r>0}\min_{||y||=r} F(y)$ and the proof is concluded. 
\end{proof}

\subsection{Lax-Milgram theorem for frames}

Let ${\mathcal F}=\{x_j\}_{j\in J}$ be a frame  for  $\H$. Let $S(x)=\sum_{j \in J}  \langle x_j, x\rangle x_j$ be the frame operator of $\F$, and let $s:\H\times \H\to\C$,
$$s(u,v)= \l Su, \, v\r.
$$   
In view of the frame inequality \eqref{eq66} and the fact that $S$ is self-adjoint,  the   bilinear form $s(\cdot,\, \cdot) $ is bounded,  coercive and symmetric.    The  Lax-Milgram and Stampacchia theorems yield    the following

\begin{theorem}
	\label{th:laxframe}
	Let $\F $  be a frame of $\H$  with frame operator $S$.  Let    $ J(u)$ be  defined as in \eqref{j1}.   
	
	a)  For a given $f\in \H$,  there exists a unique $v=S^{-1}f $ for which   $\l Sv, u\r=\l f, u\r$  for every $u\in \H$. Furthermore,   $ v $   minimizes the functional 
	\begin{equation}\label{j1} J(u)= \frac 12\l S u,\, u\r-\l f, u \r.
	\end{equation}
	
	b) If $K$ is a closed subspace of    $\H$ and  $P:\H\to\H$  is the orthogonal projection on $K$,  the  operator $PSP$  is  invertible on   $P(\H)$, and $v=(PSP)^{-1} P f$ attains the  minimum of the functional $J(u)$ on $P(\H)$.
\end{theorem}

\begin{proof}  Follows from   Corollary \ref{lm:lax2}.\end{proof}

\noindent
{\it Remarks.}
1) In  finite dimensional Hilbert spaces,  every positive self-adjoint operator is the frame operator of a certain frame  \cite{CasL}.  

2) Since  $S= T T^*$, where   $T^*: \H  \to \ell^2(J)$  is the    analysis operator of $\F $,  we can also write 
$\l Sv, u\r= \l T^*v, T^*u\r$, and 
\begin{equation}\label{j2} J(u)=  \frac 12 ||T^*u||^2   -\langle f,u\rangle. 
\end{equation}

  \section{Applications of  Lax-Milgram theorem}
 
In this section we  use   Lax-Milgram theorem to characterize scalable and piecewise scalable frames. We start with  the following

\begin{Thm} \label{L-p}
	Let   ${\mathcal F}=\{x_j\}_{j\in J}$ be a frame in $\H$   with frame constants $A<B$. Let $S$ be the frame operator of $\F$ and, for a given $f\in \H$, let 
 $ J(u)= \frac 12\l S u,\, u\r-\l f, u \r$ be as in \eqref{j1}.  
	\begin{enumerate}
		\item [a)]  We have
		$$
		-\frac{ ||f||^2} {2  A}\leq \min_{u\in \H}J(u)\leq -\frac{ ||f||^2}{2  B}.
		$$
		 
		\item [b)]    $\F$ is Parseval if and  only if,  for every $f\in\ H$
		$$ 	\min_{ u\in \H }J(u) =-\frac 12 ||f||^2. 
		$$
	\end{enumerate}
	
\end{Thm}

\begin{proof} 
	(a)  For a  given  $t\ge  0$, let $J_t$ be the restriction of $J$ to the set $\sigma_t=\{u\ :\ ||u||=t\}$.  We can see at once that  for every $u\in\sigma_t$,
	$$J_t(u)=   \left(\frac 12 \langle Su    ,\ u \rangle -\langle f, u \rangle\right)\leq  \left(\frac B2 ||u||^2 -\langle f, u \rangle\right)= \frac {Bt^2}{2}-  \langle f, u \rangle. 
	$$
By  Lemma \ref{el}, 
	\begin{equation} 
		\min_{ u\in\ H} J(u)= \inf_{t>0}\min_{ u\in \sigma_t} J_t(u) \leq  \inf_{t>0}\left(\frac {Bt^2}{2}-  \max_{ u\in \sigma_t } \langle f, u \rangle\right) 
	\end{equation}
	By Cauchy-Schwartz inequality, $\langle f, u \rangle\leq   t||f|||$  whenever $ u\in \sigma_t$, and  
	equality  is attained when $u= \frac{tf}{||f||}$. Thus, 
	$$
	\min_{ u\in\ H} J(u) \leq  \inf_{t>0}\left(\frac  {Bt^2}{2}- t||f||\right).
	$$
	The minimum of the   right-hand side  is attained when   $t= \frac{||f||}{B}$, and  so $\dsize\min_{ u\in\ H} J(u)\leq -\frac{ ||f||^2}{2B}.$
We can prove that $\dsize\min_{ u\in\ H} J(u)\ge -\frac{ ||f||^2}{2A}$ in a  similar manner.

	\medskip
	\noindent
	(b) If $\F$ is Parseval,  we have $A=B=1$ and  by (a), $\min_{ u\in\ H} J(u)=-\frac{||f||^2}{2}$.
	
	We now  assume that $\min_{ u\in\ H} J(u)= -\frac 12 ||f||$ for every $f\in\H$, and we prove that $\F$ is Parseval. 
	After perhaps re-scaling the vectors of the frame, we can assume that  $B=1$;    thus, $ \langle  Sf, \, f\rangle\leq   ||f||^2 $ for  every $f\in\H$.
	Let  $\tilde J (u)=J(u)+ \frac 12 ||f||^2 =\frac 12( \langle  Su, \, u\rangle -\l f, \ 2u-f\r )$.   By assumption,
	$\tilde J(u)\ge 0$ for all  $u\in \H$, and so  also 
	$\tilde J(f)= \langle  Sf, \, f\rangle -  ||f||^2\ge 0$.  We can infer that   $\langle  Sf, \, f\rangle =||f||^2$, and  hence that $\F$ is Parseval.
	\end{proof}

 \medskip
 Let $\F=\{ x_j\}_{j\in J} $ be a frame in $\H$; Let $P:\H\to
 H$  be an orthogonal projection, and   let   $\vec a, \, \vec b$ and $\vec c \in \ell^\infty(J)$ 
  for which  the sets 
  $\F_{\vec c} =\{c_jx_j\}_{j\in J} $  and   $\F_{\vec a, \vec b} =\{a_jPx_j+b_j(I-P)x_j\}_{j\in J}$ are frames in $\H$.   Recall that 
  $D_{\vec c}T^* $ and  $D_{\vec a}T^*P+ D_{ \vec  b} T^* (I-P)$, with $D_{\vec p}$ defined as in \eqref{dil} are     the analysis operators of $\F_{\vec c} $ and 
$\F_{\vec a, \vec b}$;   we have denoted the frame operators of $\F_{\vec c}$ and  $\F_{\vec a, \vec b}$  with $S _{\vec c} $ and  $S _{\vec a, \vec b} $. 
  
 We prove the following

 \begin{Cor} \label{L-D}

 (a)  A frame ${\mathcal F}=\{x_j\}_{j\in J}$ is scalable  if and only if there exists  $\vec c\in \ell^\infty(J)$  such that,  for every $f\in H$, 
 \begin{equation}\label{a3} 	\min_{ u\in \H }\{ \frac 12||  D_{\vec c}T^* \, u||^2  -\langle f, u \rangle\} =-\frac 12 ||f||^2.  	
\end{equation}

(b)
$\F$ is piecewise scalable   with projection $P$ if and only if there exist constants $\vec a,\ \vec b\in \ell^\infty(J)$ for which i) and ii) hold.
 
i) For every $u, w\in\H$,
 \begin{equation}\label{a2}
 \l   D_{\vec a}T ^*Pu, \  D_{\vec b}T ^*(I-P)w\r =   0, \end{equation}
  
ii) for every $f\in H$,  
\begin{equation}\label{a1}\min_{ u\in \H } \big\{ \frac 12 (|| D_{\vec a}  T^*  Pu ||^2+ || D_{\vec b}   T ^ * (I-P)u||^2) - \langle f, u\rangle\big\}=-\frac 12 ||f||^2. \end{equation}
 
\end{Cor} 
\begin{proof} 
(a) 	Since $\l   S_{\vec c} u,\, u\r =   || D_{\vec c}T^*  u||^2 $,  \eqref{a3} follows directly from Theorem \ref{L-p}.
	
(b) If $\F_{\vec a, \vec b}$ is Parseval, 
 by  Theorem \ref{L-p},   we have that 
\begin{equation}\label{bk}
 \min_{u\in\H} J(u)=\min_{u\in\H}\{\frac 12 \l S_{\vec a, \vec b} u,\, u\r-\l f,\,u\r\}= -\frac 12||f||^2.\end{equation} By  Lemma \ref{pf},   \eqref{a2} holds and 	$ S_{\vec a, \vec b}= PS_{\vec a}P+ (I-P)S_{\vec b} (I-P). $  Thus, 
\begin{equation}\label{bef}
  \l S_{\vec a, \vec b} u,\, u\r   = \l PS_{\vec a}Pu + (I-P)S_{\vec b} (I-P)u,\ u\r = || D_{\vec a}  T^*P u ||^2+ || D_{\vec b}   T ^*(I-P)  u||^2. \end{equation}
  \eqref{bk} and \eqref{bef} yield. 
  \eqref{a1}.

Assume that  \eqref{a2}  and \eqref{a1} hold.  By Lemma \ref{pf}, the frame operator of $\F_{\vec a, \vec b}$  is  $S_{\vec a, \vec b}= PS_{\vec a}P+ (I-P)S_{\vec b} (I-P)$ and 
$ || D_{\vec a}  T^*P u ||^2+ || D_{\vec b}   T ^*(I-P)  u||^2=\l S_{\vec a, \vec b}u,\, u\r$.   Thus, \eqref{a1} is equivalent to  $\min_{ u\in \H }\{ \frac 12\l S_{\vec a,b}u, \, u\r  -\langle f, u \rangle\} =-\frac 12 ||f||^2$,   	
and  by Theorem  \ref{L-p} we can conclude that $\F_{\vec a,\vec b}$ is Parseval.

	 \end{proof}

 	\begin{Cor}\label{New-scal}
 a)  A frame $ \F=\{x_j\}_{j\in J}$ is  scalable if and only if   there exist  constants $\vec c\in \ell^\infty(J)$ for which    the following inequality  holds for every $u,\, w,\, f\in \H$.
	\begin{equation}\label{newi}
		-\frac 12 ||f||^2 \leq     \frac 12 ||D_{\vec c}T^*\, u||^2 -\langle f, u \rangle   \leq   \frac 1{2 } ||u||^2 -\l f,u\r. 
	\end{equation}

b) $ \F $ is piecewise scalable with projection $P$  if and only if   there exist  constants $\vec a, \vec b\in \ell^\infty(J)$ for which   \eqref{a2}  and  the following inequality  holds for every $u,\, f\in \H$.
\begin{equation}\label{newii}
	-\frac 12 ||f||^2 \leq     \frac 12 (||D_{\vec a}T^* P\, u||^2 +||D_{\vec b}T^* (I-P)\, u||^2)  -\langle f, u \rangle   \leq   \frac 1{2 } ||u||^2 -\l f,u\r. 
\end{equation}
\end{Cor}

\begin{proof}  
   
We only prove a) since  the proof of b)  follows  from Lemma  \ref{pf} and a similar argument.

  Assume   that $\F$ is scalable.    
     Let    $\vec c\in \ell^\infty(J)$ for which    $\l S_{\vec c}u, u\r =||D _{\vec c}T^*u||^2=||u||^2$  for every $u\in \H$;
     by Corollary \ref{L-D}, for every $u,\  f\in\H$,     we have that   
  $$
  \frac 12 ||D _{\vec c}T^*u||^2  -\langle f, u \rangle   \ge \min_{u\in \H}\{\frac 12 ||D _{\vec c}T^*u||^2  -\langle f, u \rangle \}\ge -\frac 1{2 } ||f||^2.  
  $$
 and so 
    $$
    -\frac 12 ||f||^2 \leq     \frac 12 ||D _{\vec c}T^*u||^2  -\langle f, u \rangle   \leq   \frac 1{2 } ||u||^2 -\l f,u\r 
    $$
which is \eqref{newi}

 If \eqref{newi}  holds, then, when $u=f$ we have that 
$ 
-\frac 12 ||f||^2 \leq     \frac 12 \l S_{\vec c}f, f\r  -||f||^2    \leq   -\frac 1{2 } ||f||^2 
$ 
and so $\l S_{\vec c}f, f\r=||f||^2 $ for every $f\in \H$.

  \end{proof}

\section{Approximating the inverse frame operator}

Let ${\mathcal F}=\{x_j\}_{j\in J}$ be a frame for a Hilbert space $\H$ with frame operator  $S(x)=\sum_{j \in J}  \langle x_j, x\rangle x_j$.  As remarked in the introduction, every $f\in\H$  can be represented as in 
\eqref{invframe} in terms of the inverse of the frame operator. Since   evaluating $S^{-1}$ can be very difficult, even for finite frames,  it is important to   approximate $S^{-1}$, or at least to approximate the frame coefficients of $f$.  

\subsection{Approximations  with projections} 
 
For a given $f\in\H$, we let $J_f(u)=   \frac 12\l S u, u \r-\l f, u\r$.
 By Lax-Milgram theorem,   $ v=S^{-1}  f$
is the only solution of the problem   $\P$ below.
\begin{equation}
	\label{eq:problemP}
	\mbox{Problem  ${\mathcal P}$}: \qquad \mbox{  Given $f\in\H$ ,  find $v\in \H$  for which $J_f(v)\leq J_f(u)$ for every $u\in \H$.}
 %
\end{equation}

 Consider a  family of projections $P_N:\H\to \H$ such that, for every  $N\in\N$  
 \begin{equation}\label{projN}  \lim_{N\to\infty} P_N(u)=u \quad\mbox{and}  \quad P_{N+1}(\H)\supset P_N(\H)  .\end{equation}     We let $\H_N=P_N(\H)$, and   we   consider the   following 
\begin{equation}
	\label{eq:problemPN}
	\mbox{Problem  $\P_N$}: \qquad \mbox{  Given $f\in\H$,  find $v_N\in \H_N$  for which $J_f(v_N)\leq J_f(u)$ for every $u\in \H_N$.}
\end{equation}
 By Theorem \ref{th:laxframe},  $v_N=(P_NSP_N)^{-1}P_Nf$ is  the only  solution of Problem $\P_N$ in  $ \H_N$.   
 
The following theorem shows that the  solutions of the problem $\P_N$ provide a good approximation of  $v=S^{-1} f$.

\begin{Thm}\label{V2}
 Given $f\in\H$, 	let $v$ and $v_N $ be the  solutions of Problems $\P$ and  $\P_N$ defined above. Then, 
$$\lim_{N\to\infty}||v_N- v||=0.
$$
Equivalently, 
\begin{equation}\label{projf}\lim_{N\to\infty}||(P_NSP_N)^{-1} P_Nf- S^{-1} f||=0.
	\end{equation}
	\end{Thm}

From Theorem \ref{V2} follows that   for  every $f\in\H$ and every frame vector $x_j\in\F$,  
 $$\lim_{N\to\infty} \l v_N, x_j\r= \l v, x_j\r= \l S^{-1}f, x_j\r.
$$ Thus, the frame coefficients of $f$ can be approximated   in terms of the  $v_N$. 

To prove  Theorem \ref{V2},  we use the following 

\begin{theorem}(Vigier) \label{T-v}
	Let $\left\{A_n\right\}_{n \in \N}$ be a sequence  of   bounded self-adjoint operators on   $\H$. Then $\left(A_n\right)_{n \in \N}$ is strongly convergent if it is increasing and bounded above, or if it is decreasing and bounded below.   
\end{theorem}

	Here, "increasing and  bounded above " (or "decreasing and  bounded below")  means that for every $u \in\H$ with $||u||=  1$,  the sequence $n\to \l A_n u, u\r$ is increasing and bounded above  (or decreasing and bounded below).  Recall that a sequence  of bounded operators $ A_n:\H\to \H $  strongly converges  to an operator $A:\H\to\H$  if $\dsize \lim_{n\to\infty}\sup_{  ||u||=1}||A_nu -A u||=0$.

	Vigier's theorem  can also be stated in a more general form  (see \cite{Murphy2014}, Theorem 4.1).
%
Theorem \ref{T-v} yields the following 
\begin{Cor}\label{C-vig}
	 Let $\{P_N \}_{N\in\N}$   be a family of orthogonal projections  on $\H$ such that    $P_{N+1}(\H)\supset P_N(\H)$ for every $N\in\N$, and $\dsize\cup_{N=1}^\infty P_N(\H)=\H$. Then,  the projections  $P_N$   strongly converge  to the identity, i.e., 
	 $$\lim_{n\to\infty}\sup_{||u||=1} ||P_nu-u||=0.$$
\end{Cor}
\begin{proof}
Let $\H_n=P_n(\H)$.  In view of the assumptions on the $\H_n$, we have that    $\dsize \lim_{n\to\infty} P_n(x)=x$ for every $x\in\H$.	Observe that $P_{j+1}P_j= P_j P_{j+1}=P_j$, that   $P_j^2=P_j$, and $\|P_j x\|\leq \|x\|$ for all $x\in\mathcal H$ and every $j\in\N$.  Thus, 
	$$
	\langle P_j x,x\rangle = \langle P_j^2 x,x\rangle=  \|P_jx\|^2=\|P_jP_{j+1}x\|^2 \le \|P_{j+1}x\|^2=\langle P_{j+1}x,x\rangle 
	$$
	which shows that the sequence $\{P_N\}_{N\in\N}$ is increasing and bounded above. By Vigier's theorem,  the $P_N$ converge strongly to the identity, as required.
\end{proof}

 The following Lemma is Theorem 3.1-2 in \cite{Rav}  but we will prove it here for the convenience of the reader.
\begin{lemma}	\label{l:new}
	Let $f\in\H$ and let  $v$ and $v_N$  be   the  solutions  of Problems $\mathcal P$ and   $\mathcal P_N$ defined above.
There exists a constant $C>0$ independent of $n$ such that
	\begin{equation}\label{dC}
	\left\|v-v_N\right\| \leq C \inf_{u  \in \mathcal H_N}\left\|u-v \right\| .
	\end{equation}
\end{lemma}

\begin{proof}
Recall that $v$ and the $v_N$ satisfy 
$$\l Sv, u\r=\l f, u\r , \quad \l Sv_N, P_Nu\r =\l f, P_Nu\r  $$ for every $u\in\H$. Let $y\in \H_N$ and let  $w_N=y -v_N$. The element $w_N$ belongs to $\mathcal H_N$ and therefore to  $\mathcal H$.  Thus, 
$ \l Sv, \, w_N\r= \l f, w_N\r$ and $\l Sv_N, \, w_N\r= \l f, w_N\r$, and so 
 		$ \l S(v-v_N), w_N\r=0$; 
we can see at once that
		$$\l S(v-v_N),\ v -v_N\r=    \l S(v-v_N), v -y\r+\l S(v-v_N), w_N\r= \l S(v-v_N),\ v-y\r. $$
In view of  $A||x||^2\leq  \l S(x),\ x\r$ and $\l S(x), S(x)\r= \l S^2 x,\, x \r \leq B^2||x||^2$,  we gather
		$$
		A\left\|v-v_N\right\|^2 \leq  \l S(v-v_N),\ v-y\r  \leq  ||S(v-v_N)|| \,||v-y||\leq B||v-v_N||\,||v-y||
		$$
		from which follows that $\left\|v-v_N\right\|^2\leq \frac{B}{A}||v-y||.$ Since the inequality holds for every $y\in\H_N$, \eqref{dC} follows  with $C=\frac{B}{A}$. 
\end{proof}

 \begin{proof}[Proof of Theorem \ref{V2}]  By Lemma \ref{l:new},
	$$
	\left\|v-v_N\right\| \leq C \inf_{y \in \mathcal H_N}\left\|y-v \right\|\leq C||v-P_N v||
	$$
	and  by Corollary \ref{C-vig}, $\dsize\lim_{n\to\infty}\left\|v-v_N\right\|\leq C\lim_{n\to\infty}||v-P_N v||=0$
\end{proof}

\noindent
{\it Remark}.   
 Our approach to approximate the inverse frame operator is different from the approximation methods presented in \cite[Chapt 23]{Chr}. 

  In  \cite[Section 23.1]{Chr}, the author  considers the increasing sequence of  finite frames $\F_N=\{x_1,\, ...,\,  x_N\}$  in $\H_N= span\{x_1,\, ...,\, x_N\}$ and  approximates    the frame operators of $\F$ with the  frame operators of the $\F_N$. For a given $u\in\H$, the sequence 
  $S_N u=\sum_{j=1}^N \l   u, x_j\r x_j$    converges  to  $Su$,     
  but   the frame coefficients $\l S^{-1}_N u,\, x_k\r $ converge to the  $\l S^{-1}  u,\, x_k\r$ for every $u\in\H$ 
	     if and only if, for every  $n\in\N$ and  $j\leq n$, we have that  $||S^{-1}_N(x_j)|| \leq C_j$, with $C_j$ independent of $N$ (\cite[Theorem 23.1.1]{Chr}  ).
	See also \cite {christensen1993frames}.
 
	Let   $P_N$ denote the orthogonal projection on   $\H_N= Span\{x_1,\, ...,\, x_N\}$. 
 The Casazza-Christensen method  (see \cite[Section 23.2]{Chr},  and  \cite {christensen2000finite}) consists in approximating $S^{-1}$   with operators $(P_nS_{n+m(n)})^{-1} P_n :\H_n\to\H_n$, where   $m(n)>0$ is chosen so that 
the   frame  bounds of  the frames $\{P_n f_k\}_{k=1}^{n+m(n)} $ are   all the same.

 In \cite[Theorem 23.2.3]{Chr}  it is proved that $(P_nS_{n+m(n)})^{-1} P_nu$ converges  to  $S^{-1} u$  in the strong topology of $\H$,   from which  follows   that $	\lim_{n\to\infty}\l   (P_nS_{n+m(n)})^{-1} P_nu,\  x_k\r= \l   S^{-1} u,\ x_k\r $ for every $u\in\H$.
 
The method of approximation presented in our paper   relies on a family of  orthogonal projections that satisfy the assumptions in Corollary \ref{C-vig};  we   do not approximate    $S$ with   frame operators of frames   related to      $\F$ in an obvious way. In th

\section{Examples }

In the previous section we have shown that, for a given family of projections $\{P_N\}_N$ that satisfy the   the assumptions in Corollary \ref{C-vig}, the  inverse of the  frame operator $S:\H\to \H$ can be approximated arbitrarily well  (in the sense of Theorem \ref{V2})   with  the inverse of operators $P_NSP_N : P_N(\H)\to P_N(\H)$. The following example  illustrates how our results can be applied.

\noindent
{\it Example 1.}   
Let $\H=\ell^2(\N) $ and    $\F=\{f_j\}_{j=1}^\infty$, with  $f_1=e_1$ and $f_k= e_{k-1}+\frac 1k e_k$.  Here    $\{e_n\}$ is the canonical  orthonormal basis of $\ell^2(\N)$.  

It is easy to verify that the frame operator of $\F$ is represented by a matrix $M$ with elements $m_{i,j}$, with  $m_{i,j}=0$ if $|i-j|\ge 2$, and 
$m_{j,j}= 1+\frac{1}{(j+1)^2}$, and   $m_{j, j+1}=m_{j+1, j}=\frac{1}{j+1}$. 

Let $\Pi_N:\ell^2(\N)\to \ell^2(\N)$, $\Pi_N(x)=(x_1, ...,\, x_N,\,0, ...) $ be the projection on the first $N$ components of $x$.  We have observed in Section 2.1 that   $\Pi_N S\Pi_N$ is represented on $\Pi_N(\ell^2(\N))$ by the matrix  $M_N$  formed by  the intersection of the first $N$ rows and columns of $M $.  By Theorem \ref{V2},     $ \dsize \lim_{N\to \infty} ||S^{-1}f- M_N^{-1} (\Pi_N f)||=0$ whenever   $f\in \ell^2(\N)$, from which follows that 
 $\dsize 	\lim_{N\to\infty}\l   M_N^{-1} (\Pi_N f),\  x_k\r= \l   S^{-1} f,\ x_k\r$.
 
The  sub-matrices $M_N$ are symmetric tri-diagonal; the inverse  of these matrices are well-studied and explicit formulas are know. See \cite{Meu}.

 Example 23.1.3 in \cite {Chr}  shows  how the    approximation method presented in \cite[Section 23.1]{Chr}  does not work 
 for the frame $\F$.

\medskip

\subsection{Inverting an increasing sequence of matrices}

 The previous example can be generalized to any frame  of $\ell^2(J)$, with frame operator represented by the infinite matrix $M =\{m_{i,j}\}_{i,j\in\N}$. If $S$,   the frame operator of $\F$,  is represented by the matrix $M$, 
 the   operators $\Pi_N S\Pi_N$ are represented on $\H_N=\Pi_N(\ell^2(J))$ by the matrices  $M_N$  formed by  the intersection of the first $N$ rows and columns of $M $.  By Theorem \ref{V2},    $\lim_{N\to \infty} ||S^{-1}f- M_N^{-1} (\Pi_N f)||=0$ whenever   $f\in \ell^2(J)$,   and so the problem of approximating $S^{-1}$  reduces to    the problem of inverting a  sequence of matrices $\{M_N\}_N$ where, for every $N\ge 1$,  $M_N$ is the  principal sub-matrix of order $N$ of $M_{N+1}$. 
 
The results that follow  are not new, but we present  them here for completeness.

\begin{Thm}\label{SM}
	
	Let   $M_{n }$  be a symmetric  invertible $n\times n$ matrix  and let  $M_{n-1 }$ be the sub-matrix of $M_{n } $   obtained after removing the $n $-th row and column of $M_{n }$. If $k=m_{ n,n}-b^T M_{n-1}^{-1}b\ne 0$, then
	$$M^{-1}_{n }= \left(\begin{matrix} M_{n-1}^{-1}  + \frac 1k (M_{n-1}^{-1}b)(M_{n-1}^{-1}b)^T )  & \ -\frac 1k M_{n-1}^{-1} b \\ \\ -\frac 1k (M_{n-1}^{-1} b)^T & \frac 1k\end{matrix}\right)
	$$   
	\end{Thm}

To prove the theorem, we  use the Sherman–Morrison formula  (see \cite{SM}, and   also   \cite{M}).  

\begin{Lemma} 
	If the matrices  $A$ and $A+B$ are invertible and $B=uv^T$, then  
	\begin{equation}\label{mil}
		(A+B)^{-1}=A^{-1}\left(I-\frac{ uv^T A^{-1}}{1+ v^TA^{-1} u}\right).
	\end{equation}
\end{Lemma}

\begin{proof}[Proof of Theorem \ref{SM}] We can write $M_{n }=\left(\begin{matrix} M_n & b  \\ b^T& m_{n , n }\end{matrix}\right)$
where $b^T=(m_{n , 1}, ... m_{n ,\,n-1})$.  By the well-known formula for the inverse of a block matrix, we obtain 
$$M^{-1}_{n }= \left(\begin{matrix} (M_{n-1} -\frac{1}{m_{n , n }} b b^T)^{-1}& -\frac 1k M_{n-1}^{-1} b \\ \\ -\frac 1k b^TM_{n-1}^{-1}  & \frac 1k\end{matrix}\right)
$$ where $k=m_{ n,n}-b^T M_{n-1}^{-1}b$.

So, we need to evaluate  the inverse of $M_{n-1} -\frac{1}{m_{n , n }} b b^T $.

If we apply \ref{mil}    with $A=M_n$ and $u=v= m_{n,n} ^{-\frac 12}\,b  $, we obtain 
$$
(M_{n-1} -\frac{1}{m_{n , n }} b b^T )^{-1}=  M_{n-1}^{-1}\left(I + \frac 1k (bb^T M_{n-1}^{-1} ) \right), 
  $$
  and so 
 $$M^{-1}_{n }= \left(\begin{matrix} M_{n-1}^{-1}\left(I + \frac 1k (bb^T M_{n-1}^{-1} ) \right)& -\frac 1k M_{n-1}^{-1} b \\ \\ -\frac 1k b^TM_{n-1}^{-1}  & \frac 1k\end{matrix}\right)
  $$   
 \end{proof}

\subsection{Inverting the frame operator of frames with n elements in $\R^n$ }

Let $S$ be the frame operator of   frame  with $n$ elements in  $\R^n$. 
Let $A$ be an invertible $n\times n$ matrix, and let  $x_j=y_j$. Then 
\begin{equation}
 S(x)=\sum\left\langle x_j, x\right\rangle x_j  =
 \sum\left\langle A y_j, x\right\rangle A y_j  =A \sum\left\langle y_j, A^{\top} x\right\rangle y_j
\end{equation}
Thus, if we let  $S_A$ be the frame operator of the frame $A^{-1} \F$, and $y=A^Tx$, then   
$ S(x)=A S_A  A^T(x)$. If  the vectors $y_j$ are orthonormal, then the frame operator $S_A $ is the identity, and so $S =A A^T$; thus, $$S^{-1}=( A^{-1})^T A^{-1}.$$
We can construct the matrix $A$ of the change of basis from   $\{v_1,....,v_n\}$  to the orthonormal  basis $\{e_1,...,e_n\}$, which is obtained from  $\{v_1,....,v_n\}$ by the Gram-Schmidt process. Using the notation 
$ \langle u|v\rangle :=\frac{\langle u, v \rangle}{\langle u, u \rangle}$, we can construct recursively the columns of the matrix $A$ as follows. 
Since 
$ 
e_1=\frac{1}{\left\|v_1\right\|} \cdot v_1 \quad \Rightarrow \quad  v_1=\left\|v_1\right\| \cdot e_1 $,  and the first column of $A$ is 
 $C_1=\left(\begin{array}{c}
\left\|v_1\right\| \\
0 \\
\vdots \\
0
\end{array}\right).
$ 
   
Let us now assume that we already have the orthonormal vectors $e_1, \ldots, e_k$ and the first $k$ columns of   $A$.   We let 
$ 
w_{k+1}=v_{k-1}-\left\langle e_1 \mid v_{k+1}\right\rangle e_1-\left\langle e_2 \mid v_{k+1}\right\rangle e_2-\ldots-\left\langle e_k \mid v_{k+1}\right\rangle e_k \quad $
and $   e_{k+1}:=\frac{1}{\left\|w_{k+1}\right\|} w_{k+1}
$.
Then,
$$
v_{k+1}=\left\langle e_1 \mid v_{k+1}\right\rangle e_1+\left\langle e_2 \mid v_{k+1}\right\rangle e_2+\ldots+\left\langle e_i \mid v_{kk+1}\right\rangle e_k+\left\|w_{k+1}\right\| e_{k+1}
$$
gives the $k+1$-th column of $A$:
$$
C_{k+1}=\left(\begin{array}{c}
\left\langle e_1 \mid v_{k+1}\right\rangle \\
\vdots \\
\left\langle e_i \mid v_{k+1}\right\rangle \\
\left\|w_{k+1}\right\| \\
0 \\
\vdots \\
0
\end{array}\right)
$$
 The matrix of the operator $S$ has the following form
 	$$S=\left(\begin{array}{cccc}
 		\left\|v_1\right\|  & \left\langle e_1 \mid v_{2}\right\rangle & \cdots & \left\langle e_1 \mid v_{n}\right\rangle\\
 		0 & \left\|w_{2}\right\| & \cdots & \left\langle e_2 \mid v_{n}\right\rangle \\
 		\vdots & \vdots & \ddots & \vdots \\
 		0 & 0 & \cdots & \left\|w_{n}\right\|
 	\end{array}\right)$$
 	It is possible to factorize the matrix $S$ as follows:
 	$$D=\operatorname{Diag}\left(\left\|v_1\right\|, \left\|w_{2}\right\|,\ldots,\left\|w_{n}\right\|\right), \quad S=DQ$$
 	$$Q=\left(\begin{array}{cccc}
 		1  & \frac{\left\langle e_1 \mid v_{2}\right\rangle}{\left\|v_1\right\|} & \cdots & \frac{\left\langle e_1 \mid v_{n}\right\rangle}{\left\|v_1\right\|}\\
 		0 & 1 & \cdots & \frac{\left\langle e_2 \mid v_{n}\right\rangle}{\left\|w_{2}\right\|} \\
 		\vdots & \vdots & \ddots & \vdots \\
 		0 & 0 & \cdots & 1
 	\end{array}\right)=I+N$$
 	where $N$ is nilpotent, i.e. $N^n=O$. Therefore:
 	$$
 	Q^{-1}=(I+N)^{-1}=I+\sum_{k=1}^{n-1}(-1)^k N^k=I-N+N^2-N^3+\cdots+(-1)^{n-1} N^{n-1},
 	$$
 	Finally:
 	$$
 	S^{-1}=Q^{-1} D^{-1}=Q^{-1} \operatorname{Diag}\left(\frac{1}{\left\|v_1\right\|}, \frac{1}{\left\|w_2\right\|}, \ldots, \frac{1}{\left\|w_n\right\|}\right)
 	$$


\begin{thebibliography}{1}
\expandafter\ifx\csname url\endcsname\relax
  \def\url#1{\texttt{#1}}\fi
\expandafter\ifx\csname urlprefix\endcsname\relax\def\urlprefix{URL }\fi
\expandafter\ifx\csname href\endcsname\relax
  \def\href#1#2{#2} \def\path#1{#1}\fi

\bibitem{B} P. Balasz, {\it Matrix Representation of Operators Using Frames}, Sampling Theory in Signal and Image Processing  7   (2008) pp. 39-–54 

\bibitem {Brezis} H.  Brezis,  {\it Functional analysis, Sobolev spaces and partial differential equations},  Springer Science+Business Media, LLC 2011

\bibitem{CC}  J. Cahill and X. Chen, {\it A note on scalable frames}, Proceedings of the 10th International Conference on Sampling Theory and Applicationsp (2013) pp. 93-96.

\bibitem{CDT} P. Cazazza, L. De Carli and T. Tran, {\it Pieceiwse scalable frames}, Preprint (2022).  https://arxiv.org/pdf/2203.12678.pdf

\bibitem{CASAZZA2000338}
P.~G. Casazza, O.~Christensen, {\it Approximation of the inverse frame operator and
applications to Gabor frames}, Journal of Approximation Theory 103~(2) (2000)
pp. 338--356.


\bibitem{CasL}
P.~G. Casazza, M.~Leon, {\it Existence and construction of finite frames with a
given frame operator}, Int. J. Pure Appl. Math 63~(2) (2010) pp. 149--157.

\bibitem{CCH}  P. G. Casazza and X. Chen, {\it Frame scalings:  A condition number approach}, Linear Algebra and Applications. {\bf523} (2017) pp. 152-168. 

\bibitem{CK}  A. Chan, R. Domagalski, Y. H. Kim, S. K. Narayan, H. Suh, and X. Zhang, {\it Minimal scalings and structual properties of scalable frames}, Operators and matrices. {\bf11}(4) (2017)  1057-1073.
 

\bibitem{CKO}  X. Chen, G. Kutyniok, K. A. Okoudjou, F. Philipp, and R. Wang, {\it Measures of scalability},
IEEE Trans. Inf. Theory. {\bf61}(8) (2015) pp. 4410-4423.



\bibitem{christensen1993frames}
O.~Christensen, {\it Frames and the projection method}, Applied and Computational
Harmonic Analysis 1~(1) (1993) pp. 50--53.

\bibitem{christensen2000finite}
O.~Christensen, {\it Finite-dimensional approximation of the inverse frame operator},
Journal of Fourier Analysis and Applications 6~(1) (2000)  pp. 79--91.


\bibitem{Chr}
O.~Christensen,   {\it An introduction to frames and Riesz bases} (second edition)  
Birkh\"auser, 2016.

\bibitem{christensen2005finite}
O.~Christensen, T.~Strohmer,  {\it The finite section method and problems in frame
theory}, Journal of Approximation Theory 133~(2) (2005) pp. 221--237.

\bibitem{Coh1}
A.~Cohen, W.~Dahmen, R.~DeVore,  {\it Adaptive wavelet methods for elliptic operator
equations: convergence rates}, Math. Comp. 70~(233) (2001) pp. 27--75.

\bibitem{Coh2}
Cohen, Dahmen, DeVore, {\it Adaptive wavelet methods ii---beyond the elliptic case,}
Foundations of Computational Mathematics 2~(3) (2002) pp. 203--245.

\bibitem{Coh3}
A.~Cohen, W.~Dahmen, R.~DeVore,  {\it Adaptive wavelet schemes for nonlinear
variational problems}, SIAM Journal on Numerical Analysis 41~(5) (2003)
pp. 1785--1823.

\bibitem{Coh4}
A.~Cohen, W.~Dahmen, R.~Devore, {\it Sparse evaluation of compositions of functions
using multiscale expansions}, SIAM Journal on Mathematical Analysis 35~(2)
(2003) pp. 279--303.


\bibitem{CKL}  M. S. Copenhaver, Y. H. Kim, C. Logan, K. Mayfield, S. K. Narayan, M. J. Petro, and J. Sheperd, {\it Diagram 
	vectors and tight frame scaling in finite dimensions}, Oper. Matrices.
{\bf8}(1) (2014) pp. 73-88.


\bibitem{DK}  R. Domagalski, Y. Kim, and S. K. Narayan, {\it On minimal scalings of scalable frames},
Proceedings of the 11th International Conference on Sampling Theory and Applica-
tions. (2015) pp. 91-95.

\bibitem{Fei}
H.~Feichtinger, P.~Jorgensen, D.~Larson, G.~Olafsson,  {\it Mini-workshop: Wavelets
and frames}, in: mini-workshop held February, Vol. 1521, 2004, pp. 479--543.

\bibitem{H} R. Horn, C. Johnson, {Matrix Analysis,
	Second Edition}, Cambridge University Press (2013)

\bibitem{KO}  G. Kutyniok, K.A. Okoudjou, F. Phillip, and E.K. Tuley, {\it Scalable frames}, 
Linear
Algebra Appl. {\bf 438}(5) (2013) pp. 2225-2238.


\bibitem{KOP}  G. Kutyniok, K. A. Okoudjou, and F. Philipp, {\it Scalable frames and convex geometry},
Contemp. Math. {\bf626} (2014) pp. 19-32.

\bibitem{M} Miller, K. S.
{\it On the inverse of the sum of matrices.}
Math. Mag. 54 (1981), no. 2, pp. 67--72

\bibitem{Meu} G. Meurant, {\it A review on the inverse of symmetric tridiagonal and block tridiagonal matrices}, SIAM J. Matrix  Anal. Appl.
Vol. 13 (1992) no. 3, pp. 707--728 

\bibitem{Murphy2014}
Murphy, Gerald J. {\it C*-algebras and operator theory}. Academic press, 2014.


\bibitem{Rav}
P.-A. Raviart, J.-M. Thomas, P.~G. Ciarlet, J.~L. Lions, {\it Introduction {\`a}
l'analyse num{\'e}rique des {\'e}quations aux d{\'e}riv{\'e}es partielles,}
Vol.~2, Dunod Paris, 1998.

\bibitem {SM} Sherman, Jack; Morrison, Winifred J.  {\it Adjustment of an Inverse Matrix Corresponding to a Change in One Element of a Given Matrix}. Annals of Mathematical Statistics Vol. 21 (1950) n. 1, pp. 124--127. 


\bibitem{You}
R.~M. Young, {\it An introduction to non-harmonic {F}ourier series}, Academic Press,
  2001.



 
 
 
 
 
 
 
  
  

\end{thebibliography}
\end{document}